\documentclass[12pt,leqno]{amsart}
\numberwithin{equation}{section}
\usepackage{amssymb}
\usepackage{amsmath}
\usepackage{latexsym}

\def\R{\mathbb R}
\def\C{\mathbb C}

\def\N{\mathbb N}
\def\im{\operatorname{Im}}
\def\re{\operatorname{Re}}

\def\arg{\operatorname{arg}}

\def\meas{\operatorname{meas}}
\def\sing{\operatorname{sing}}

\def\deg{\operatorname{deg}}
\def\diam{\operatorname{diam}}
\def\dens{\operatorname{dens}}
\def\area{\operatorname{area}}

\newtheorem{lemma}{Lemma}[section]
\newtheorem{theorem}{Theorem}[section]

\theoremstyle{definition}

\theoremstyle{remark}

\newtheorem{remark}{Remark}[section]

\def\be{\begin{enumerate}}
\def\ee{\end{enumerate}}

\title[Julia sets of positive measure]
{Entire functions with Julia sets of positive measure}
\thanks{Both authors were supported by the Research Training
Network CODY  of the European Commission and the G.I.F., the
German--Israeli Foundation for Scientific Research and Development,
Grant G-809-234.6/2003. The second author was also supported by
the ESF Research Networking Programme
HCAA}
\author{Magnus Aspenberg}
\address{Mathematisches Seminar,
Christian--Albrechts--Universit\"at zu Kiel,
Lude\-wig--Meyn--Str.~4, D--24098 Kiel, Germany}
\email{aspenberg@math.uni-kiel.de, bergweiler@math.uni-kiel.de}
\author{Walter Bergweiler}
\begin{document}
\begin{abstract}
Let $f$ be a transcendental entire function for which the set of critical and
asymptotic values is bounded.
The Denjoy-Carleman-Ahlfors
theorem implies that if the set of all $z$ for which
$|f(z)|>R$ has $N$ components for some $R>0$, then 
the order of $f$ is at least $N/2$.
More precisely, we have
$\log\log  M(r,f) \geq \frac12 N \log r-O(1)$,
where $M(r,f)$ denotes the maximum modulus of $f$.
We show that if $f$ does not grow much faster than this, then 
the escaping set and the Julia set of $f$ have positive Lebesgue measure.
However, as soon as the order of $f$ exceeds $N/2$, this 
need not be true.
The proof requires a sharpened form of an estimate of Tsuji related 
to the Denjoy-Carleman-Ahlfors
theorem.
\end{abstract}

\maketitle

\section{Introduction and results}

The Julia set $J(f)$ of an entire function is defined as
the set of all
points in $\C$ where the iterates $f^n$ of $f$ do not form a
normal family; see~\cite{Bergweiler93} for an introduction
to transcendental dynamics.

McMullen ~\cite{McMullen87} proved that $J(\sin (\alpha z+\beta))$ has
positive Lebesgue measure 
and that $J(\lambda e^z)$ has
Hausdorff dimension~$2$,
for $\alpha, \beta, \lambda\in \C$, $\alpha,\lambda\neq 0$.
The result on the Hausdorff dimension of $J(\lambda e^z)$ has been
extended to large classes of functions; see~\cite{Baranski08,Bergweiler08a,
Schubert07,Tan03}. It is the purpose of this paper to
exhibit a class of functions whose Julia sets have positive
measure. However, we begin by briefly describing the results on Hausdorff
dimension.

We first recall that the Eremenko-Lyubich class $B$ consists of
all entire functions for which the set of finite asymptotic values
and critical values is bounded. Eremenko and Lyubich~\cite{Eremenko92}
proved that if $f\in B$, then the escaping set $I(f)$ consisting
of all points $z\in \C$ for which $f^n(z)\to \infty$ is contained
in $J(f)$. In fact, it follows that $J(f)=\overline{I(f)}$ for
$f\in B$. It is easily seen that $\sin(\alpha z+\beta)\in B$ and
$\lambda e^z\in B$. 
McMullen actually proved that $I(\sin(\alpha z+\beta))$ has
positive measure and $I(\lambda e^z)$ has Hausdorff dimension 2.

Next we note that the order $\varrho(f)$ of an entire function $f$
is defined by
\[\varrho(f)=\limsup_{r\to\infty}\frac{\log\log M(r,f)}{\log r}\]
where
\[M(r,f)=\max_{|z|=r}|f(z)|.\]
Thus $\varrho(f)$ is the infimum of the set of all $\mu$
such that $|f(z)|\leq\exp(|z|^\mu)$ for large~$|z|$, with
$\varrho(f)=\infty$ if no such $\mu$ exists. We note that
$\varrho(\lambda e^z)=\varrho(\sin(\alpha z+\beta))=1$.

McMullen's result on the Hausdorff dimension of $J(\lambda e^z)$
was substantially generalized by Bara\'nski~\cite{Baranski08} and,
independently, Schubert~\cite{Schubert07}. They proved that if $f\in B$
and $\varrho(f)<\infty$, then $J(f)$ has Hausdorff dimension 2. 
The special case where $f$ has the form
\begin{equation}\label{aa}
f(z) = \int_0^zP(t) e^{Q(t)} dt + c,
\end{equation}
with polynomials $P$ and $Q$ and with $c\in\C$ had been treated before
by Taniguchi~\cite{Tan03}.
These functions are in $B$ and we have $\varrho(f)=\deg P$.

A generalisation of the result of  Bara\'nski and Schubert
was given in~\cite{Bergweiler08a} where it is
shown that if $f\in B$ and
\[q=\limsup_{r\to\infty}\frac{\log\log\log M(r,f)}{\log \log
r}<\infty,\] then $I(f)$ and hence $J(f)$ have Hausdorff dimension
at least $(q+1)/q$.
For further results on the Hausdorff dimension of Julia sets
of entire functions we refer to~\cite{Baranski08a,Bergweiler08}
and, in particular, the survey~\cite{Stallard08}.

While McMullen's result on the Hausdorff dimension of $J(\lambda e^z)$ thus 
has prompted a lot of further research, 
there seem to be no papers whose main intention is to extend McMullen's result 
that $J(\sin(\alpha z+\beta))$ has positive measure to more general classes of functions.
However, there are some papers devoted to ergodic properties of transcendental 
entire and meromorphic functions and their results in particular imply that the
Julia sets of certain functions have positive measure.
We mention the work of Bock~\cite{Bock98} whose results imply that if
$f\in B$ and if 
there exist $\alpha>0$ and $R>0$ such that
the set of all $z$ for which $|z|>R$ and $|f(z)|<e^t$ 
is contained in finitely many 
domains of the form $\{z: |\arg z -s|\leq t/(\log|z|)^\alpha\}$ for all large $t$,
then $I(f)$ has positive measure.
For example, this result applies to $f(z)=R(e^z)$, where $R$ is a rational 
function with $R(0)=R(\infty)=0$, or to $f(z) = \sin P(z)$, where $P$
is a polynomial. 
Skorulski~\cite{Skorulski05} considered functions of the 
form
\[
f(z)=\frac{a\exp(z^p)+b\exp(-z^p)}{c\exp(z^p)+d\exp(-z^p)},
\]
where $p\in\N$ and $a,b,c,d\in\C$,
and 
Hemke~\cite{Hemke05} studied a class which contains all functions
of the form~\eqref{aa}.
Both Skorulski and Hemke 
proved that $J(f)$ has positive measure for the functions considered, if the 
singularities of the inverse have a certain behavior under iteration.
For a more detailed description of the above and other results on the measure 
of Julia and escaping sets we refer to the survey by 
Kotus und Urba\'nski~\cite[section~7]{Kotus08}.

We shall exhibit  a condition which depends only on the growth of~$f$
and which,
for $f\in B$, implies that $I(f)$ and $J(f)$ have positive measure.
Before stating this condition we recall that (one version of) the
Denjoy-Carleman-Ahlfors-Theorem~(see~\cite[p.~173]{Goldberg08},
\cite[section~8.3]{Hayman89}
or~\cite[p.~309]{Nevanlinna53}) says that if $f$ is
entire, $R>0$ and $N$ denotes the number of components of
\[
A_R=\{z\in \C:|f(z)|>R\},
\]
then $N\leq \max\{1,2\varrho (f)\}$.
As we shall see below, we have $\varrho(f)\geq\frac{1}{2}$ for
$f\in B$. (This seems to have been observed first
in~\cite{Bergweiler95,Langley95};
see also~\cite[Lemma~3.5]{Rippon05}.)  Thus
$N\leq 2\varrho(f)$ in this case. More precisely, we even 
have (see~\cite[Theorem 8.9]{Hayman89} or~\cite[p.~312]{Nevanlinna53})
\begin{equation}\label{Narzisse}
\log\log M(r,f)\geq\frac{N}{2}\log r -O(1)
\end{equation}
as $r\to \infty$. We shall show that if $f\in B$ does not grow
much faster than guaranteed by~\eqref{Narzisse}, then $I(f)$ and
$J(f)$ have positive measure. In particular, we shall see that this
is the case if we have equality in~\eqref{Narzisse} or, more generally, if
\[\log\log M(r,f)\leq
\left(\frac{N}{2}+\frac{1}{\log^m r}\right) \log r\] 
for large
$r$, where $\log ^m$ denotes the $m$-th iterate of the logarithm.

To formulate a more precise condition we fix $\beta\in(0,1/e)$ and
note that the function $E_\beta(x)=e^{\beta x}$ has a repelling
fixed point $\xi>e$ with multiplier
\[
\mu=E'_\beta(\xi)=\beta
E_\beta(\xi)=\beta\xi>1.
\]
 Now Schr\"oder's functional equation
\begin{equation}\label{Tulpe}
\Phi(E_\beta(z))=\mu\Phi(z)
\end{equation}
has a unique solution $\Phi$ holomorphic in a neighborhood of
$\xi$ and satisfying $\Phi(\xi)=0$ and $\Phi'(\xi)=1$. It is
not difficult to see that $\Phi$ is real on the real axis and that
$\Phi$ has a continuation $\Phi:[\xi,\infty)\to [0,\infty)$ so
that~\eqref{Tulpe} is satisfied for $\xi\leq z\leq\infty$.
Moreover, $\Phi$ is increasing on the interval $[\xi,\infty)$
and we have $\lim_{x\to\infty}\Phi(x)=\infty$ while
\[\lim_{x\to\infty} \frac{\Phi(x)}{\log^m x}=0\]
for all $m\in \N$. Thus $\Phi$ tends to $\infty$, but slower than
any iterate of the logarithm. Hence the function
\begin{equation}\label{Rose}
\varepsilon:(\xi,\infty)\to(0,\infty),\quad
\varepsilon(x)=\frac{1}{\Phi(x)}
\end{equation}
is decreasing and tends to 0 as $x\to \infty$, but it tends to 0
slower than any of the functions $1/\log^m x$.
We mention that the function $\Phi$ also appears in recent work
of Peter~\cite{Peter} on Hausdorff measure of exponential Julia sets.

\begin{theorem} \label{theorem1}
Let $f\in B$ and suppose that $A_R$ has $N$ components for some $R>0$. Let
$\varepsilon(x)$ be defined by~\eqref{Rose} and suppose that
\[\log\log M(r,f)\leq\left(\frac{N}{2}+\varepsilon(r)\right)\log r\]
for large $r$. Then $I(f)$ and $J(f)$ have positive Lebesgue measure.
\end{theorem}
In section~\ref{example} we will give an example which shows
that the function $\varepsilon(r)$ in Theorem~\ref{theorem1}
cannot be replaced by a positive constant~$\varepsilon$.

The proof of Theorem~\ref{theorem1} will use some ideas connected to the
Denjoy-Carleman-Ahlfors-Theorem. One way to prove the latter
theorem is based on an estimate of Tsuji~\cite[p.~116]{Tsuji59}. To formulate
this result, let $U$ be a component of $A_R$, let
\begin{equation} \label{thetar}
\theta(r)=\meas\left(\left\{t\in [0,2\pi]:re^{it}\in U
\right\}\right),
\end{equation}
put $\theta^*(r)=\theta(r)$ if $\{z\in
\C:|z|=r\}\not\subset U$ and put $\theta^*(r)=\infty$ and thus
$1/\theta^*(r)=0$ otherwise. Tsuji's result says that for $0<\kappa <1$
there exist constants $C$ and $r_0$ such that
\begin{equation}\label{Nelke}
\log\log M(r,f)\geq\pi\int_{r_0}^{\kappa r}\frac{dt}{t\theta^*(t)}-C\end{equation}
for $r>r_0/\kappa $.

Eremenko and Lyubich~\cite{Eremenko92} proved that if $f\in B$ and $R$
is chosen such that all critical and finite asymptotic values have
modulus less than $R$, then all components of $A_R$ are simply
connected and unbounded. For large $r$ we thus have
$\theta^*(r)=\theta(r)\leq 2\pi$ and hence~\eqref{Nelke} yields
\[\log\log
M(r,f)\geq\frac{1}{2}\int_{r_0}^{\kappa r}\frac{dt}{t}-C=\frac{1}{2}\log
r -C-\log \frac{\kappa }{r_0}\] for large $r_0$ and $r>r_0/\kappa $. In
particular, it follows that $\varrho(f)\geq \frac{1}{2}$ for $f\in
B$, as mentioned above.

To prove Theorem \ref{theorem1} we shall need a refinement of~\eqref{Nelke} in
the case that
\begin{equation}\label{Aloe}\{z\in \C:|z|=r\}\not\subset U\end{equation}
and hence $\theta^*(r)=\theta(r)$ for large $r$.
\begin{theorem} \label{theorem2}
Let $f$ be entire, $R>0$ and let $U$ be a component of $A_R$ such
that~\eqref{Aloe} holds for all large~$r$.
Let $0<\beta<\frac{1}{2}$ and put
\[
V=\left\{z\in U:|f(z)|\geq \exp\left(|z|^\beta\right)\right\}
\]
and $\psi(r)=\meas(\{t\in [0,2\pi]:r e^{it}\in V\})$. Then
for $0<\kappa <1$ there exist constants $C$ and $r_0$ such
\[\log\log M(r,f)\geq \pi\int^{\kappa r}_{r_0}\frac{dt}{t\psi(t)}-C\]
for $r\geq r_0/\kappa $.
\end{theorem}
We shall prove
Theorem~\ref{theorem2} in section~\ref{proof2} and
Theorem~\ref{theorem1} in section~\ref{proof1}.

\section{Proof of Theorem~\ref{theorem2}}  \label{proof2}

In this section we prove Theorem \ref{theorem2} following an argument
by Tsuji~\cite[section~III.17]{Tsuji59}; see 
also~\cite[section 5.1]{Goldberg08} and,
for a slightly different approach,~\cite[section 8.1]{Hayman89}. 
The original result is stated for subharmonic functions and the main
difference here is that our function $\log |f(z)| - |z|^{\beta}$ is not
subharmonic. 

The following lemma \cite[p. 112]{Tsuji59} is known
as Wirtinger's inequality.
\begin{lemma}

Let $f$ and $f'$ be continuous in $[a,b]$ and $f(a)=f(b)=0$. Then 
\[
\int_a^b f'(x)^2 dx \geq \frac{\pi^2}{(b-a)^2} \int_a^b f(x)^2 dx.
\]
\end{lemma}
Let $v(z)= \log|f(z)| -
|z|^{\beta}$ and let 
\[
V = \{ z \in U : v (z) \geq 0 \} \subset U,
\]
where
$U$ is a component of $A_R$. Define 
$V_{r} = \{ \theta \in [t_r,2\pi+t_r] : r e^{i \theta} \in V \}$ 
where $t_r$ is chosen such that $r e^{i t_r}\notin U$ and put
\[
m(r) =m_v(r) =\frac{1}{2\pi}  \int_{V_r} v(re^{it})^2 dt .
\]
Hence $\meas (V_r)  = \psi(r)$. Recall that $\theta(r)$ is
defined by (\ref{thetar}). 

Next we prove the following. 
\begin{lemma} \label{lowerbound}
There exist positive constants $c$ and $r_0$ such that $m(r) \geq c
r$ for $r \geq r_0$.
\end{lemma}

\begin{proof}
Let $u(z) = \log |f(z)|-\log R$ for $z \in U$
and $u(z)=0$ outside $U$. Then $u \geq 0$ and $u$ is subharmonic. Let
\[
m_u(r) = \frac{1}{2\pi} \int_0^{2\pi} u (re^{i\theta})^2 d\theta.
\]
Note that $\log M(r,f)\geq \sqrt{m_u(r)}$.
Inequality (\ref{Nelke}) is actually a corollary of a more general 
result~\cite[Theorem 8.2]{Hayman89} which says that
\[
\log \sqrt{m_u(r)} \geq \pi  \int_{r_0}^{\kappa r} \frac{dt}{t \theta
  (t)} - C,
\]
for $r > r_0 / \kappa$. Since $\theta ( t) \leq 2\pi$ we obtain
\[
\log m_u(r)  \geq 2 \pi \int_{r_0}^{\kappa r} \frac{dt}{t \theta (t)}
- C \geq \int_{r_0}^{\kappa r} \frac{dt}{t} - C \geq \log r - O(1),
\]
and we conclude that $m_u(r) \geq c' r$ for some $c'>0$, if $r > r_0/\kappa$. 

To obtain a similar estimate for $m(r)$, first write
\[
\int_{V_r} v(r e^{i \theta})^2 d \theta =
\int_{V_r} u(r e^{i \theta})^2 d \theta 
- 2 \int_{V_r} u(r e^{i \theta}) r^{\beta} d \theta +
\int_{V_r} r^{2 \beta} d \theta. 
\]
By the Cauchy-Schwarz inequality, 
\[
\int_{V_r} u(r e^{i \theta}) r^{\beta} d \theta \leq \sqrt{
  \int_{V_r} u(r e^{i \theta})^2 d \theta } \sqrt{
  \int_{V_r} r^{2 \beta} d \theta  } 
  \leq \sqrt{2\pi} r^{\beta}
   \sqrt{ \int_{V_r} u(r e^{i \theta})^2 d \theta }.
\]
Hence 
\begin{equation}
\begin{aligned}
\int_{V_r} v(r e^{i \theta})^2 d \theta &\geq
\int_{V_r} u(r e^{i \theta})^2 d \theta - \sqrt{8 \pi}
r^{\beta} \sqrt{
  \int_{V_r} u(r e^{i \theta})^2 d \theta } 
\\
&=
\sqrt{
  \int_{V_r} u(r e^{i \theta})^2 d \theta } \bigg( \sqrt{
  \int_{V_r} u(r e^{i \theta})^2 d \theta } - \sqrt{8
  \pi} r^{\beta} \bigg) .
  \label{sqrts}
\end{aligned}
\end{equation}
We have
\[
\begin{aligned}
c' r 
&\leq m_u(r)\\
&=
\frac{1}{2\pi} \int_{V_r} 
u(r e^{i \theta})^2 d \theta +
\frac{1}{2\pi} \int_{\{\theta: \; 0 \leq u(r e^{i \theta}) < r^{\beta}\}} u(r e^{i \theta})^2 d \theta \\ 
&\leq
\frac{1}{2\pi} \int_{V_r} u(r e^{i \theta})^2 d \theta + r^{2
  \beta}. 
\end{aligned}
\]
Hence 
\[
\int_{V_r} u(r e^{i \theta})^2 d \theta \geq 2\pi c' r - 2 \pi
r^{2 \beta} \geq c'' r 
\]
for some $c''>0$ and $r \geq r_0$, provided $r_0$ is large enough.

Hence (\ref{sqrts}) yields
\[
m(r)
=\frac{1}{2\pi}
\int_{V_r} v(r e^{i \theta})^2 d \theta 
\geq \frac{1}{2\pi} \sqrt{c'' r}
(\sqrt{c'' r} - \sqrt{8 \pi} r^{\beta} ) 
\geq c r,
\]
for some $c>0$ and $r \geq r_0$, if $r_0$ is sufficiently large.
\end{proof}

Next we prove (following Tsuji \cite{Tsuji59}) that $m(r)$ is a convex function of $\log r$
for large $r$.
\begin{lemma}
There exist $r_0 > 0$ such that
\[
\frac{d^2 m(r)}{d (\log r)^2} \geq 0,
\]
for all $r \geq r_0$.
\end{lemma}

\begin{proof}
The Laplacian in polar coordinates is given by 
\begin{equation}  \label{laplace}
\frac{1}{r^2} \bigg( \frac{\partial^2 v(r e^{i \theta})}{\partial (\log r)^2 } +
\frac{\partial^2 v(r e^{i \theta})}{\partial \theta^2} \bigg) =   \Delta v(r e^{i \theta}) =  \Delta \left(\log |f(r e^{i \theta})| - r^{\beta} \right)
= -\beta^2 r^{\beta-2}. 
\end{equation}
The set $V_r$ consists of finitely many intervals 
$[\alpha_j(r),  \beta_j(r) ]$. Then 
\begin{align}
\frac{d m(r)}{ d \log r} &= \frac{1}{2 \pi} \sum_j \frac{d}{d \log r}
\int_{\alpha_j(r)}^{\beta_j(r)} v(r e^{i \theta})^2 d \theta \nonumber
\\
&=
\frac{1}{2 \pi} \sum_j \int_{\alpha_j(r)}^{\beta_j(r)} \frac{\partial
  v (r e^{i \theta})^2}{\partial \log r } d \theta + v(re^{i
  \beta_j(r)})^2 \frac{d \beta_j(r)}{d \log r} - v(re^{i
  \alpha_j(r)})^2 \frac{d \alpha_j(r)}{d \log r} \nonumber \\
&= 
\frac{1}{2 \pi} \sum_j \int_{\alpha_j(r)}^{\beta_j(r)} \frac{\partial
  v (r e^{i \theta})^2}{\partial \log r } d \theta \nonumber \\
&=
\frac{1}{\pi} \int_{V_r} v(re^{i \theta}) \frac{\partial v(re^{i
    \theta}) }{\partial \log r} d \theta, \nonumber 
\end{align}
since $v(re^{i\alpha_j(r)}) =v(re^{i\beta_j(r)})=0$. 

Also, 
\begin{equation} \label{mbis}
\frac{d^2 m(r)}{d (\log r)^2} 
=  \frac{1}{\pi} \int_{V_r} \bigg( \frac{\partial v(re^{i\theta})}{ \partial \log r} \bigg)^2 
+ v(re^{i\theta}) \frac{\partial^2 v(re^{i\theta})}{\partial (\log r)^2} d \theta.
\end{equation}
Now 
\[
\frac{\partial^2}{\partial \theta^2} \left(v(re^{i\theta})^2\right) 
= \frac{\partial }{ \partial \theta} \bigg( 2 v(re^{i\theta}) \frac{\partial v(re^{i\theta})}{\partial \theta} \bigg) = 2 v(re^{i\theta}) \frac{\partial^2 v(re^{i\theta})}{\partial \theta^2}  + 2 \bigg( \frac{\partial v(re^{i\theta})}{\partial \theta} \bigg)^2,
\]
and since $v(re^{i\alpha_j(r)}) =v(re^{i\beta_j(r)})=0$ we have
\[
\int_{\alpha_j(r)}^{\beta_j(r)} \frac{\partial^2}{\partial \theta^2} (v(re^{i\theta})^2) d \theta = \bigg[ 2v(re^{i\theta}) \frac{\partial v(re^{i\theta})}{\partial \theta} \bigg]_{\alpha_j(r)}^{\beta_j(r)} = 0
\]
for all $j$. Thus 
\[
\int_{V_r} v(re^{i\theta}) \frac{\partial^2 v(re^{i\theta})}{ \partial \theta^2} d \theta = -
\int_{V_r} \bigg( \frac{\partial v(re^{i\theta})}{\partial \theta} \bigg)^2 d \theta
\]
and using (\ref{laplace}) and (\ref{mbis}) we obtain
\[
\frac{d^2 m (r)}{d (\log r)^2} 
= \frac{1}{\pi} \int_{V_r} 
\bigg( \bigg( \frac{\partial v(re^{i\theta})}{\partial \log r} \bigg)^2 
+ \bigg( \frac{\partial v(re^{i\theta})}{\partial \theta} \bigg)^2 
- v(r e^{i\theta}) \beta^2 r^{\beta} \bigg) d \theta.
\]

Let us write
\begin{align}
J_1 &= \frac{1}{\pi} \int_{V_r} \bigg( \frac{\partial
  v(re^{i\theta})}{\partial \log r} \bigg)^2  d \theta, \nonumber \\
J_2 &= \frac{1}{\pi} \int_{V_r} \bigg( \frac{\partial
  v(re^{i\theta})}{\partial \theta} \bigg)^2 d \theta, \nonumber \\
J_3 &= \frac{1}{\pi} \int_{V_r} v(r e^{i\theta}) \beta^2 r^{\beta} d
\theta.  \nonumber 
\end{align}
To estimate $J_1$ we use the Cauchy-Schwarz inequality to obtain
\[
\begin{aligned}
\bigg( \frac{d m(r)}{d \log r} \bigg)^2 
&= \bigg( \frac{1}{\pi} \int_{V_r} v(re^{i\theta}) \frac{\partial v(re^{i\theta})}{ \partial\log r}  d \theta \bigg)^2  \\
&\leq \frac{1}{\pi^2}  \int_{V_r}  v(re^{i\theta})^2 d \theta        
\int_{V_r} \bigg( \frac{\partial v(re^{i\theta})}{\partial \log r} \bigg)^2 d \theta \\
&\leq 2 m(r) J_1.
\end{aligned}
\]
Hence
\[
J_1 \geq \frac{1}{2 m(r)} \bigg( \frac{d m(r)}{d \log r} \bigg)^2.
\]
Recall that $V_r$ is a union of intervals $[\alpha_j(r),\beta_j(r)]$
so that $\psi(r) = \sum_j \left(\beta_j(r)-\alpha_j(r)\right)$. 
Using Wirtinger's inequality on each of these intervals we get
\[
\begin{aligned}
\frac{1}{\pi} \int_{\alpha_{j}(r)}^{\beta_j(r)} \bigg( \frac{\partial
  v(re^{i\theta})}{\partial \theta} \bigg)^2 d \theta 
  &\geq
\frac{\pi}{\left(\beta_j(r)-\alpha_j(r)\right)^2} \int_{\alpha_{j}(r)}^{\beta_j(r)} v(r e^{i
  \theta})^2 d \theta
  \\
  &
  \geq  \frac{\pi}{\psi(r)^2} \int_{\alpha_{j}(r)}^{\beta_j(r)} v(r e^{i
  \theta})^2 d \theta .
  \end{aligned} 
\]
Summing over all $j$ yields
\[
\begin{aligned}
J_2 
& = \frac{1}{\pi}\sum_j \int_{\alpha_{j}(r)}^{\beta_j(r)} \bigg( \frac{\partial
 v(re^{i\theta})}{\partial \theta} \bigg)^2 d \theta \\
& \geq 
\frac{\pi}{\psi(r)^2} \sum_j \int_{\alpha_{j}(r)}^{\beta_j(r)} 
v(r e^{i\theta})^2 d \theta \\ 
& = 
\frac{\pi}{\psi(r)^2} 
\int_{V_r} v(r e^{i \theta})^2 d \theta \\
&= \frac{2
  \pi^2}{\psi(r)^2}         m(r).
\end{aligned}
\]
To estimate $J_3$ we use the Cauchy-Schwarz inequality again to obtain 
\begin{align}
J_3 &= \frac{\beta^2 r^{\beta}  }{\pi} \int_{V_r} v( r e^{i \theta })
d \theta \nonumber \\
&\leq \frac{\beta^2 r^{\beta}  }{\pi} 
 \bigg( \int_{V_r} v(r e^{i \theta})^2 d \theta
\bigg)^{1/2}  \bigg( \int_{V_r} 1 \, d \theta \bigg)^{1/2}
\nonumber \\
&\leq 2\beta^2 r^{\beta}  \sqrt{ m( r)} \nonumber\\
&= 2\beta^2 r^{\beta - 1/2}
r^{1/2} \sqrt{ m(r)} \nonumber 
\end{align}
Using Lemma~\ref{lowerbound} and putting $\gamma=2\beta^2/\sqrt{c}$
we obtain
\[
J_3\leq \gamma  r^{\beta-1/2} m(r).
\]

Hence
\begin{equation}
\begin{aligned}
\frac{d^2 m(r)}{d (\log r)^2} &\geq J_1 + J_2 - J_3
\\
&\geq  
\frac{1}{2 m(r)} \bigg( \frac{d m (r)}{d \log r}
\bigg)^2 + \frac{2\pi^2}{\psi(r)^2} m( r) - \gamma r^{\beta -
  1/2} m(r) \\
&= \frac{1}{2 m(r)} \bigg( \frac{d m (r)}{d \log r} \bigg)^2 + 
\frac{m(r)}{2} \bigg( \bigg( \frac{2\pi}{\psi(r)} \bigg)^2 
- 2 \gamma r^{\beta - 1/2} \bigg). \label{m-est}
\end{aligned}
\end{equation}
Since $\psi(r) \leq 2 \pi$ and $\beta<\frac12$ there is some $r_0 > 0$ such that 
\begin{equation}
 \label{alpha>0}
\bigg( \frac{2\pi}{\psi(r)} \bigg)^2 - 2 \gamma r^{\beta - 1/2} \geq 0\quad \text{for all $ r \geq r_0$}.
\end{equation}
Hence 
\[
\frac{d^2 m(r)}{d (\log r)^2} \geq 0\quad \text{for all $r \geq r_0$}. 
\]
\end{proof}

By~\eqref{alpha>0} we may define $\alpha$ and $\tilde{\alpha}$ by
\[
\tilde{\alpha} (r)
=\alpha (\log r)
=\sqrt{ \left( \frac{2\pi}{\psi(r)} \right)^2 - 2 \gamma r^{\beta - 1/2}}
\]
for $r \geq r_0$.
Put 
$\mu(t) = m(e^t)$. We use
the change of variables $r = e^t$, so 
$\alpha (t)=\tilde{\alpha} ( r)$ and $\mu(t) = m (r)$. Now inequality (\ref{m-est}) becomes
\[
\mu''(t) \geq \frac{\mu'(t)^2}{2 \mu(t)} + \frac{1}{2} \alpha(t)^2 \mu(t). 
\]
 From this we deduce (see Tsuji~\cite[pp. 114-115]{Tsuji59}) that
\begin{equation} \label{mu-alpha}
\bigg( \frac{\mu''(t)}{\mu'(t)} \bigg)^2  \geq \alpha(t)^2. 
\end{equation}

We now argue that in fact also $\mu'(t) \geq 0$ for large enough
$t$. Lemma \ref{lowerbound} implies that $\mu(t) = m(e^t) \geq c e^t$
for all $t \geq \log r_0$. This means that $\mu'(t) = r m'(r) >
0$ for some $t$ because otherwise $\mu$ would be bounded. 
Since also 
$\mu''(t) = d^2 m(r)/d (\log r)^2\geq 0$ 
for large $t$ this implies that actually $\mu'(t) > 0$ for
all large $t$, say $t \geq \log r_0$. 

Hence from (\ref{mu-alpha}) we get 
\[
\frac{\mu''(t)}{\mu'(t)}  \geq \alpha(t)\quad \text{for all $ t\geq \log r_0$}. 
\]
To conclude the proof of Theorem  \ref{theorem2}, let $ \tau > t_0 = \log r_0$ and note that
\[
\log \mu'(\tau) - \log \mu'(t_0) = \int_{t_0}^{\tau} \frac{\mu''(\rho)}{\mu'(\rho)} d \rho 
\geq \int_{t_0}^{\tau} \alpha (\rho) d \rho, 
\]
so
\[
\mu'(\tau) \geq \mu'(t_0) \exp \bigg\{ \int_{t_0}^{\tau} \alpha(\rho) d\rho  \bigg\}.
\]
With $t = \log r > t_0$ we have,
since $\mu (t)$ is increasing for $t \geq \log r_0$,
\begin{equation}
\mu(t) \geq \mu(t) - \mu(t_0) = \int_{t_0}^{t} \mu'(\tau) d \tau \geq
\mu'(t_0) \int_{t_0}^{t}  \exp \bigg\{ \int_{t_0}^{\tau} \alpha(\rho) d\rho \bigg\}d\tau .
\end{equation}
With $\rho= \log s$ and $\tau = \log \sigma$ we get 
\[
\mu(t) \geq \mu'(t_0) \int_{r_0}^r \exp \bigg\{ \int_{r_0}^{\sigma} 
\frac{\tilde{\alpha} (s)}{s} ds \bigg\} \frac{d \sigma}{\sigma}.
\] 
For $r \geq r_0/\kappa $, 
with $0 < \kappa < 1$, we thus have 
\[
\mu(t) \geq \mu'(t_0) \int_{\kappa r}^r \exp \bigg\{ \int_{r_0}^{\sigma} 
\frac{\tilde{\alpha} (s)}{s} ds \bigg\} \frac{d \sigma}{\sigma} \geq 
\mu'(t_0)  (1-\kappa) \exp \bigg\{  \int_{r_0}^{\kappa r} \frac{\tilde{\alpha} (s)}{s} ds   \bigg\}.
\]
With $c_0 = \mu'(t_0)$ thus
\begin{equation} \label{mittel-ziel}
\mu(t)\geq c_0(1-\kappa) \exp \bigg\{ \int_{r_0}^{\kappa r}  
\frac{\tilde{\alpha} (s)}{s} ds  \bigg\}.
\end{equation}
We want to estimate the integral on the right side.
We have 
\[
\tilde{\alpha}(s) = \frac{2 \pi}{ \psi(s)} \sqrt{ 1- \frac{\psi(s)^2}{2\pi^2} \gamma s^{\beta - 1/2} } \geq \frac{2 \pi}{\psi(s)} \bigg( 1 - \frac{\psi(s)^2}{2\pi^2} \gamma s^{\beta - 1/2} \bigg)
=\frac{2 \pi}{ \psi(s)}- \frac{\gamma}{\pi}\psi(s) s^{\beta - 1/2},
\]
for $s\geq r_0$, where we used that $\sqrt{x} \geq x$ for $x \in [0,1]$. 
Therefore,
\[
\int_{r_0}^{\kappa r}\frac{\tilde{\alpha} (s)}{s} ds
\geq 
2 \pi \int_{r_0}^{\kappa r}  \frac{ ds }{s \psi(s)} 
-  \frac{\gamma}{\pi} \int_{r_0}^{\kappa r} \psi(s) s^{\beta-3/2} ds
\]
But, since $\beta < 1/2$ and $\psi(s) \leq 2 \pi$,
\[
\frac{\gamma}{\pi} \int_{r_0}^{\kappa r} \psi(s)
s^{\beta - 3/2} ds    \leq c_1
\]
for some constant $c_1 > 0$.
Hence~\eqref{mittel-ziel} yields
\[
\mu(t)\geq c_0(1-\kappa) e^{-c_1} \exp \bigg\{  2 \pi \int_{r_0}^{\kappa r}  \frac{ ds }{s \psi(s)} ds     \bigg\}.
\]
Recalling that $t=\log r$ and $m(r)=\mu(t)$, we get
\[
m(r) \geq c_2 \exp \bigg\{ 2 \pi \int_{r_0}^{\kappa r}  \frac{ ds }{s \psi(s)} ds       \bigg\},
\]
where $c_2=c_0(1-\kappa) e^{c_1}$.

From this and the fact that 
\[
\log \log M(r,f) \geq \log \max\limits_{|z|=r} v(z) \geq \log \sqrt{m(r)} = \frac{1}{2} \log m(r), 
\]
Theorem \ref{theorem2} follows.

\begin{remark}
With some more effort (see again \cite{Tsuji59}), one can 
show that
\[
m(\rho) \leq \frac{2 e^{c_1+1} m(r)}{1- \kappa} 
\exp \bigg\{ - 2 \pi \int_{\rho}^{\kappa r} 
\frac{ \tilde{\alpha} (s) } { s } ds \bigg\}
\]
for $r_0\leq \rho< \kappa r$.
Using this it follows that the constant $C$ in Theorem 1.2 only depends 
on $\kappa$, $r_0$ and $\beta$.
\end{remark}

\section{Proof of Theorem~\ref{theorem1}}  \label{proof1}
Before we begin with the proof of Theorem~\ref{theorem1} we need 
some auxiliary results. 
We begin by describing the  {\it logarithmic change of variable}
which was the main tool used by Eremenko and Lyubich to study the dynamics
of a function $f\in B$.
We choose $R>|f(0)|$ such that $\Delta_R=\{z \in \C: |z| > R\}$
contains no critical values and no asymptotic values of~$f$. As already
mentioned in the introduction, Eremenko and Lyubich showed that
every component $U$ of $A_R=f^{-1}(\Delta_R)$ is simply connected. The map
$f:U \to \Delta_R$ is thus a universal covering. With
$H=\{z\in\C:\re z> \log R\}$ the map $\exp: H\to \Delta_R$ is
also a universal covering and so there exists a biholomorphic map
$G:U\to H$ such that $f=\exp\circ G$.
This construction can be done
for every component $U$ of $A_R$ and putting
$W=\exp^{-1}(A_R)$ we can thus define
$F:W\to H$, $F(z)=G(e^z)$. Thus $\exp F(z) =
f(e^z)$ and $F$ maps every component of $W$ univalently
onto~$H$. We say that $F$ is the function obtained from $f$ by a
logarithmic change of variable.

If $\phi$ is a branch of the inverse
function of $F$ and if $w\in H$, then
$\phi$ is defined in particular in the disk of radius
$\re w- \log R$ around $w$. Thus Koebe's one
quarter theorem implies that  $\phi(H)$ contains a disk
of radius $\frac14|\phi'(w)|(\re w- \log R)$ around $\phi(w)$.
Since $W$ and hence $\phi(H)$
do not contain vertical segments of
length greater than $2\pi$, and thus in particular no disc of
radius greater than $\pi$, it follows that
\begin{equation} \label{1a}
|\phi'(w)|\leq \frac{4\pi}{\re w- \log R}.
\end{equation}
In terms of $F$ this inequality takes the form
\begin{equation} \label{1b}
|F'(z)|\geq \frac{\re F(z)- \log R}{4\pi}
\end{equation}
for $z\in W$.

Another tool we shall use is the Besicovitch covering
lemma~\cite[Theorem 3.2.1]{deGuzman81}. Here we denote the
ball of radius $r$ around a point $x\in \R^n$ by $B(x,r)$.
\begin{lemma} \label{besicovitch}
Let $K\subset \R^n$ be bounded and
$r:K\to(0,\infty)$. Then there exists an at most countable subset $L$
of $K$
satisfying
\[
K\subset \bigcup_{x\in L} B(x,r(x))
\]
such that no point in $\R^n$ is contained in more than $4^{2n}$
of the balls $B(x,r(x))$, $x\in L$.
\end{lemma}

We now begin with the proof of   Theorem~\ref{theorem1}. 
Let $U_1,U_2,\ldots U_N$ be the components of $\{z\in
\C:|f(z)|>R\}$. We may assume that $R$ is so large that
$E^n_\beta(x)\to \infty$ as $n\to \infty$ for $x>\log R$. For
$j=1,\ldots,N$ we put
\[V_j=\left\{z\in U_j:|f(z)|\geq \exp \left(|z|^\beta\right)\right\}\]
and denote 
by $\psi_j(r)$ the measure of the set of
all $t\in[0,2\pi]$ such that $r e^{it}\in V_j$. It follows from
Theorem~\ref{theorem2} that for $0<\kappa <1$ there exist constants 
$r_0$ and $C $ such that
\[\log\log M(r,f)\geq \pi
\int_{r_0}^{\kappa r}\frac{dt}{t\psi_j(t)}-C \] for $r>r_0/\kappa $. Hence
\[
\log\log M(r,f)\geq \pi\int^{\kappa r}_{r_0}\left(\frac{1}{N}\sum^N_{j=1}
\frac{1}{\psi_j(t)}\right)\frac{dt}{t}-C .
\] 
By the
Cauchy-Schwarz inequality we have
\[N^2=\left(\sum^N_{j=1}
\frac{\sqrt{\psi_j(t)}}{\sqrt{\psi_j(t)}}\right)^2
\leq\left(\sum^N_{j=1}
\frac{1}{\psi_j(t)}\right)
\cdot
\left(\sum^N_{j=1} \psi_j(t)\right).\] 
With
\[\psi(t)=\sum^N_{j=1}\psi_j(t)\]
we deduce that
\[\sum^N_{j=1}\frac{1}{\psi_j(t)}\geq\frac{N^2}{\psi(t)}\]
and hence that
\[\log\log M(r,f)\geq N\pi\int^{\kappa r}_{r_0}\frac{1}{\psi(t)}
\frac{dt}{t}-C .\] 
By hypothesis we have
\[\log\log M(r,f)\leq\frac{N}{2}\log r +\varepsilon(r)\log
r=\frac{N}{2} \bigg( \int^{\kappa
  r}_{r_0}\frac{dt}{t}+\log\frac{r_0}{\kappa } \bigg) +\varepsilon(r)\log
r.\] It follows 
from the last two inequalities 
that
\[
\begin{aligned}
  \varepsilon(r)\log r+C + \frac{N}{2} \log\frac{r_0}{\kappa } & \geq
  N\pi\int^{\kappa r}_{r_0}\left(\frac{1}{\psi(t)}-\frac{1}{2\pi}\right)\frac{dt}{t}\\
   & = N\pi\int^{\kappa r}_{r_0}
   \frac{2\pi-\psi(t)}{2\pi\psi(t)}\frac{dt}{t}\\
   & \geq\frac{N}{4\pi}\int^{\kappa r}_{r_0}(2\pi-\psi(t))\frac{dt}{t}. 
\end{aligned}\]
Since $\varepsilon(r)$ is decreasing and $\varepsilon(r)\log r\to
\infty$ as $r\to \infty$ we obtain
\begin{equation}\label{2a}
\int^r_{r_0}(2\pi-\psi(t))\frac{dt}{t}
\leq\frac{4\pi}{N}\left(
\varepsilon\left(\frac{r}{\kappa }
\right)\log\frac{r}{\kappa }+C +\frac{N}{2} \log\frac{r_0}{\kappa }\right)\leq\frac{5\pi}{N}\varepsilon(r)\log r
\end{equation} 
for large $r$.

Let now $F$ be the function obtained from $f$ by the logarithmic
change of variable. With $W=\bigcup^N_{j=1} \exp^{-1}(U_j)$ and
$H=\{z\in \C:\re z>\log R\}$ we thus have 
\[
F:W\to H,\ F(z)=\log f(e^z),
\]
for some branch of the logarithm.
Moreover, $F$ maps every component of $W$ bijectively onto $H$.

The real part of $F$ is large on the set $L=\bigcup^N_{j=1}\exp
^{-1}(V_j)$. In fact,
\[L=\{z\in W:\re F(z)\geq\exp(\beta\re z)\}.\]
We put
\[T=\{z\in L:F^n(z)\in L\mbox{ for all }n\in \N\}.\]
For $z\in T$ we then have
\[\re F^n(z)\geq E^n_\beta(\re z)\]
and thus Re $F^n(z)\to\infty$ as $n\to\infty$.
It follows that
\[|f^n(e^z)|=\exp(\re F^n(z))\to \infty\]
so that $\exp(T)\subset I(f)$.

We shall show that $\area(T)>0$. This then implies that
$\area(I(f))>0$.
In order to prove that $\area(T)>0$ we consider 
for $n\geq 0$ the set
\[T_n=\{z\in L:F^k(z)\in L\mbox{ for }0\leq k\leq n\}\]
so that $T_0=L$. Then
\[T=\bigcap^\infty_{n=1}T_n.\]
Let $S=\C\backslash L$ 
and put
\[\Psi(x)=\meas\left\{y\in[0,2\pi]:x+iy\in S\right\}\]
for $x>\log R$. Since for $x+iy\in L$ we have $e^xe^{iy}\in
\bigcup^N_{j=1} V_j$ it follows that
\[\Psi(x)=2\pi-\psi(e^x).\]
From~\eqref{2a} we deduce that
\[
   \int^x_{\log r_0} \Psi(s)ds
=\int^{e^x}_{r_0}\Psi (\log t)\frac{dt}{t}
= \int^{e^x}_{r_0}(2\pi-\psi(t))\frac{dt}{t}
\leq \frac{5\pi}{N}\varepsilon(e^x)x .
\]
We put $\delta(x)=\varepsilon(e^x)$. It follows that if $x_0\geq
\log r_0$, then
\[\int^x_{x_0}\Psi (s) ds\leq \frac{5\pi}{N}\delta(x) x.\]
For $z=x+iy\in \C$ with $x>2\log R$ we denote by $Q(z)$ the square
of sidelength $x$ centered at $z$. Thus
\[Q(z)=\left\{\zeta\in \C:|\mbox{Re }\zeta-x|\leq\frac{1}{2}x,|\im\zeta - y|\leq\frac{1}{2}x\right\}.\] 
Now
\[
\area\left(\left\{z\in S:\{x_1\leq \re z\leq x_2, y_0\leq \im z\leq y_0+2\pi\right\}\right)
=\int^{x_2}_{x_1}\Psi(s)ds\]
for $\log R<x_1<x_2$ and $y_0\in \R$. Since $Q(z)$ can be covered
by $[\frac{x}{2\pi}+1]$ horizontal strips of width $2\pi$ we
obtain
\begin{equation}\label{2a1}
  \area(Q(z)\cap S) 
  \leq\left(\frac{x}{2\pi}+1\right)\int_{\frac{1}{2}x}^{\frac{3}{2}x}\Psi(s)ds
\leq\left(\frac{x}{2\pi}+1\right)\frac{5\pi}{N}\delta \left(\frac{3}{2}x\right)\frac{3}{2}x
\leq \frac{4}{N}\delta(x) x^2 
\end{equation}
for large $x$. 
Recall that for measurable sets $A,B\subset\C$ the density of $A$ in $B$ is defined by
\[
\dens (A,B)=\frac{\area(A\cap B)}{\area(B)}.
\]
With this notation~\eqref{2a1} takes the form
\begin{equation}\label{2f}
\dens (S,Q(z))
\leq
\frac{4}{N}\delta (x)
\end{equation}

We now fix $n\in\N$ and consider $u\in T_{n-1}\backslash T_n$ with
$\re u>x_0$ for some large number $x_0$ to be determined later.
Put $v=F^n(u)$ and $x_n=E_\beta^n(x_0)$, where $E_\beta^n(x)=e^{\beta x}$.
Then
\[\re v\geq E_\beta^n (\re u)\geq x_n.\]

A standard argument (cf. Remark~\ref{distortion} at the end of this section)
using Koebe's distortion theorem shows that for large $u$ and $v$
the branch $\phi_n$ of the inverse function of $F^n$ which satisfies
$\phi_n(v)=u$ extends to a univalent map on $B(v,\frac{3}{4}\re v)$
and thus has bounded distortion on $Q(v)$. It follows that there
exists a constant $K$ such that
\begin{equation}\label{2b}
\dens (\phi_n(Q(v)\cap S),\phi_n(Q(v)))\leq K\dens
(S,Q(v))\leq\frac{4K}{N}\delta(x_n).
\end{equation}
Moreover, Koebe's theorem yields that there exist positive
constants $\sigma,\tau$ such that if
\[r_n(u)=|\phi_n'(v)|\cdot \re v = \frac{\re F^n(u)}{|(F^n)'(u)|},\]
then
\begin{equation}\label{2c}
B(u,\sigma r_n(u))\subset \phi_n(Q(v))\subset B(u,\tau r_n(u)).
\end{equation}
It can be deduced from~(\ref{1b}) and the chain rule that
\begin{equation}\label{2c1}
r_n(u)\leq 5\pi
\end{equation}
if $x_0$ is sufficiently large.
From~\eqref{2b} and~\eqref{2c} we can deduce that
\begin{equation}\label{2d}
\dens (F^{-n}(S),B(u,\tau
r_n(u))\leq\frac{4K}{N}\left(\frac{\tau}{\sigma}\right)^2\delta (x_n).
\end{equation}

We now fix $w_0$ with $\re w_0>2x_0$ and consider the square
$P=Q(w_0)$. Suppose that $n\in \N$ and
\begin{equation}\label{2e}
\dens (T_{n-1},P)\geq \frac{1}{2}.
\end{equation}
By Lemma~\ref{besicovitch}, we can find 
an at most countable subset $A$ of $T_{n-1}\cap P$ such that the disks $B(u,\tau r_n(u))$, $u\in A$, 
cover $T_{n-1}\cap P$, 
with no point being covered more than $4^4$ times. With
\[
P'=\left\{z\in \C:|\re(z-w_0)|<\frac{1}{2}\re
w_0+5\pi\tau,\im(z-w_0)|<\frac{1}{2}\re w_0+5\pi\tau\right\}
\] 
we have
$B(u,\tau r_n(u))\subset P'$ for all $u\in A$ by~\eqref{2c1}, and for large $x_0$ we also have
$\area(P')\leq 2 \area(P)$.

We now deduce from~\eqref{2d} and~\eqref{2e} that
\[\begin{aligned}
 \area (F^{-n}(S)\cap T_{n-1}\cap P)
  \leq & \area \left(F^{-n}(S)\cap \bigcup_{u\in A} B(u,\tau r_n(u))\right)\\
  \leq & \sum_{u\in A}\area (F^{-n}(S)\cap B(u,\tau r_n(u))) \\
  \leq & \frac{4K}{N}\left(\frac{\tau}{\sigma}\right)^2\delta(x_n)\sum_{u\in A}\area (B(u,\tau r_n(u))) \\
  \leq & \frac{4K}{N}\left(\frac{\tau}{\sigma}\right)^2 4^4\delta(x_n)\area(P') \\
  \leq & \frac{8K}{N}\left(\frac{\tau}{\sigma}\right)^2 4^4\delta(x_n)\area(P)  \\
  \leq & \frac{16K}{N}\left(\frac{\tau}{\sigma}\right)^2 4^4\delta(x_n)\area(T_{n-1}\cap P) . 
\end{aligned}\]
With 
\[
\eta =\frac{16K}{N}\left(\frac{\tau}{\sigma}\right)^2 4^4
\]
we thus have
\[
\dens(F^{-n}(S), T_{n-1} \cap P)
\leq \eta \delta(x_n).
\] 
Since
$F^{-n}(S)\cap T_{n-1}=T_{n-1}\backslash T_n$ we obtain
\[\dens (T_{n-1}\backslash T_n, T_{n-1} \cap P)\leq \eta \delta (x_n)\]
and thus
\[\dens(T_n, T_{n-1} \cap P)\geq 1-\eta \delta(x_n).\]
Induction shows that
\begin{equation}\label{2g}
\dens (T_n,T_0 \cap P)\geq \prod^n_{k=1}(1-\eta \delta(x_k)),
\end{equation}
as long as
\begin{equation}\label{2h}
\dens (T_k,P)\geq\frac{1}{2}\quad \text{for}\ k\leq n-1.
\end{equation}
Now
\[
  \delta(x_n) 
= \delta(E_\beta^n(x_0))=\varepsilon (\exp(E_\beta ^n (x_0))
\leq \varepsilon(E_\beta^{n+1}(x_0))=\frac{1}{\Phi\left(E^{n+1}_\beta(x_0)\right)} 
=\frac{1}{\mu^{n+1}\Phi(x_0)}. 
\]
We conclude that the infinite product
$\prod^\infty_{k=1}(1-\eta \delta(x_k))$ converges and by choosing
$x_0$ large we may achieve that
\begin{equation}\label{2i}
\prod^\infty_{k=1}(1-\eta \delta(x_k))\geq\frac{3}{4}.
\end{equation}
From~\eqref{2f} we deduce that
\begin{equation}\label{2j}
\dens (T_0,P)=1-\dens(L,P)\geq\frac{2}{3}
\end{equation}
for large $w_0$.

Suppose now that~\eqref{2h} and hence~\eqref{2g} holds for some
$n\in\N$. Then, since $T_n\subset T_0$,
\[\dens(T_n,P)=\dens(T_n,P\cap T_0)\cdot \dens(T_0,P)\geq
\frac{3}{4}\cdot \frac{2}{3}=\frac{1}{2}\] 
by~\eqref{2i} and~\eqref{2j}.

Thus~\eqref{2h} and hence~\eqref{2g} hold with $n-1$ replaced by~$n$.
Induction thus shows that~\eqref{2g} holds for all
$n\in\N$. It follows that
\[\dens(T,P \cap T_0)\geq \prod^\infty_{k=1}(1-\eta \delta(x_k))\geq
\frac{3}{4}.\] In particular, $\area(T)>0$.

\begin{remark}\label{distortion}
We used in the proof that the branch $\phi_n$ of the inverse
function of $F^n$ which maps $v=F^n(u)$ to $u$ extends to a 
univalent map on $B(v,\frac34 \re v)$. In order to see this
we note that if $\phi$ is the branch of $F^{-1}$ which maps 
$v$ to $F^{n-1}(u)$, then $\phi$ is univalent in $H$ and 
it follows from~\eqref{1a} that 
\[
\diam\phi\left( B(v,\tfrac34 \re v)\right) 
 \leq \tfrac32 \re v \max_{w\in B(v,\frac34 \re v)} |\phi'(w)|\\
 \leq \tfrac32 \re v \frac{4\pi}{\frac14 \re v -\log R}\\
 \leq 48\pi
\]
if $ \re v < 8 \log R$. 
We conclude that if $u$ and hence  $F^{n-1}(u)$ are large enough, then 
\[
\phi\left( B(v,\tfrac34 \re v)\right)
\subset
B\left( F^{n-1}(u), \tfrac34  F^{n-1}(u)\right).
\]
The above claim now follows by induction.

Essentially the same argument can be found, e.g., in~\cite{Baranski08,Bergweiler08a}.
The argument gets much simpler if the postsingular set
\[
P(f)=
\overline{\bigcup_{n=0}^\infty f^n\left(\sing\left(f^{-1}\right)\right)}
\]
is bounded. (Here $\sing\left(f^{-1}\right)$ denotes the set of 
singularities of the inverse function of~$f$.)
We note that if $f\in B$ and $f_\lambda(z)=\lambda f(z)$, then
$P(f_\lambda)$ is bounded for small~$\lambda$.
A theorem of Rempe~\cite{Rempe} implies that there exists $R_\lambda>0$
such that $f$ and $f_\lambda$ are quasiconformally conjugate on
the set $\{z: |f^n(z)|\geq R_\lambda\ \text{for all}\ n\geq 0\}$.
Since quasiconformal mappings map sets of positive area to 
sets of positive area, the conclusion for $f$ follows from
that for~$f_\lambda$. Thus it actually suffices to consider the
special case that $P(f_\lambda)$ is bounded.
\end{remark}

\begin{remark}
Let $f$ be a function meromorphic in the plane which has $N$ logarithmic
singularities over infinity. Denote by $U_1,U_2,\dots,U_N$ the corresponding
logarithmic tracts.
With 
\[
U=\bigcup_{j=1}^N U_j
\quad\text{and}\quad
M_U(r,f)=\max_{|z|=r,\, z\in  U}|f(z)|
\]
we deduce from the proof that 
the conclusion of Theorem~\ref{theorem1}
holds if
\[\log\log M_U(r,f)\leq\left(\frac{N}{2}+\varepsilon(r)\right)\log r\]
for large $r$.
It follows from standard estimates of Nevanlinna theory~\cite[Theorem~7.1]{Goldberg08} that 
$\log M_U(r,f)\leq 3 m(2r,f)$. Using this it is not difficult to see 
that the conclusion of Theorem~\ref{theorem1} holds if 
\[\log m(r,f)\leq\left(\frac{N}{2}+\varepsilon(r)\right)\log r\]
and thus, in particular, if 
\[\log T(r,f)\leq\left(\frac{N}{2}+\varepsilon(r)\right)\log r.\]
The dynamics of meromorphic functions with logarithmic singularities are studied for example
in~\cite{Baranski08a,Bergweiler08}.
\end{remark}

\section{An example} \label{example}

We consider Mittag-Leffler's function
\[E_\alpha(z)=\sum^\infty_{n=0} \frac{z^n}{\Gamma(\alpha n +1)}\]
for a parameter $\alpha\in (0,2)$. It satisfies the following
conditions: \be\item [(i)] $\varrho(E_\alpha)=\frac{1}{\alpha}$
\item[(ii)] $E_\alpha$ is bounded in the sector $\left\{re^{it}:
r>0, |t-\pi|\leq \left(1- \tfrac{1}{2}
\alpha\right)\pi\right\}$ \item[(iii)] $E_\alpha\in B$ \ee

Properties (i) and (ii) are well-known; see,
e.g.,~\cite[pp. 83-86]{Goldberg08}. Since we could not find a
proof of (iii) in the literature, we indicate an argument to prove
(iii) below. 

It follows from (ii) and (iii) and a theorem of Eremenko and
Lyubich~\cite[Theorem~7]{Eremenko92} that $\area (I(E_\alpha))=0$. Moreover, the
arguments yield (cf.~\cite[Theorem~8]{Eremenko92}) that if $\lambda>0$ is
chosen so small that the Fatou set of $\lambda E_\alpha$ consists
of a single, completely invariant attracting basin, then
$\area (J(\lambda E_\alpha))=0$.

We see that there exist functions $f\in B$ whose order is
arbitrarily close to $\frac{1}{2}$ such that $\area (I(f))=\area (J(f))=0$.
Considering $f(z)=E_\alpha(z^N)$ we obtain functions
where $A_R$ has $N$ components and where $\varrho(f)$ is close to
$\frac{1}{2}N$. Thus the function $\varepsilon(r)$ in Theorem~\ref{theorem1}
cannot be replaced by a positive constant $\varepsilon$.

Let us now prove 
property (iii).
From~\cite[pp. 84-85]{Goldberg08} we get the following
representation for $E_{\alpha}$, where $\varrho = 1/ \alpha$:
\begin{align}
E_{\alpha}(z) &= w_1(z) \quad \text{for $\tfrac12\alpha \pi <
  | \arg (z) | \leq \pi$},  \label{Lakritz} \\
  E_{\alpha}(z) &= w_2(z) + \varrho \exp\left(z^{\varrho}\right)
\quad \text{for $|\arg (z)| \leq 
\tfrac12\alpha \pi + \delta$}, \label{Bonbon} 
\end{align}
where $0<\delta\leq \max\left\{\frac12 \alpha\pi,\left(1-\frac12 \alpha\right)\pi\right\}$ and $w_i(z) = O(1/|z|)$ as $|z|\to \infty$, for $i=1,2$. 
Note that Properties~(i) and (ii) follow immediately from~\eqref{Lakritz}
and~\eqref{Bonbon}. 

To prove Property (iii), put $S_\delta = \{ z : |\arg (z) | \leq \frac12
\alpha\pi + \delta \}$. 
For $z \in S_{\delta/2}$ we have $B(z,|z|\sin(\delta/2)) \subset
S_\delta$ and Cauchy's formula yields
\begin{equation}\label{Keks}
\left| E_{\alpha}'(z) - \varrho^2 z^{\varrho-1} \exp\left(z^{\varrho}\right)\right|
 =  |w_2'(z)| 
 =\left|\frac{1}{2\pi i}
 \int_{\partial B\left(z,|z|\sin\left(\frac{\delta}{2}\right)\right)} 
 \frac{w_2(\zeta)}{(z-\zeta)^2}d\zeta\right|
 =O\left( \frac{1}{|z|^2}\right)
\end{equation}
as $|z|\to\infty$, uniformly in $z \in S_{\delta/2}$.
For $z \in \C \setminus S_{\delta/2}$ we have $B(z,|z|\sin(\delta/2))  \subset \C \setminus S_0$
 and in the same way Cauchy's formula yields
\[
|E_\alpha'(z)| = |w_1'(z)| =O\left( \frac{1}{|z|^2}\right)
\]
as $|z|\to\infty$, uniformly in $\C\setminus S_{\delta/2}$.

We now show that the set of critical values of $E_\alpha$ is
bounded. Since $E_\alpha$ is bounded in~$\C \setminus S_0$ we have to
consider only the critical points in $S_0$. So let
$\xi\in S_0$ be a critical point of $E_\alpha$; that is,
$E'_\alpha(\xi)=0$. Then
\[\varrho^2|\xi|^{\varrho-1}\left|\exp\left(\xi^\varrho\right)\right|
\leq
\frac{C_1}{|\xi|^2}\] for some constant $C_1$ by (\ref{Keks})
and thus
\[|E_\alpha(\xi)|\leq
\varrho
\left|\exp\left(\xi^\varrho\right)\right|
+\frac{C_2}{|\xi|}\leq\frac{C_1}{\varrho|\xi|^{\varrho+1}}+\frac{C_2}{|\xi|^2}\]
for some constant $C_2$ by (\ref{Bonbon}). It follows that the set
of critical values of $E_\alpha$ is bounded. Since $E_\alpha$ has
only finitely many asymptotic values by the
Denjoy-Carleman-Ahlfors-Theorem, it follows that $f\in B$.
(Actually the only asymptotic value of $E_\alpha$ is 0. This can
be deduced from (\ref{Lakritz}) and (\ref{Bonbon}).)

\end{document}